\documentclass{amsart}
\usepackage{graphicx}
\newtheorem{thm}{Theorem}[section]
\newtheorem{cor}[thm]{Corollary}
\newtheorem{lem}[thm]{Lemma}

\theoremstyle{definition}
\newtheorem{defn}[thm]{Definition}
\theoremstyle{remark}
\newtheorem{rem}[thm]{Remark}
\numberwithin{equation}{section}

\begin{document}

\title[On Moore-Yamasaki-Kharazishvili type measures and  infinite powers $\cdots$]{On Moore-Yamasaki-Kharazishvili type measures and  the infinite powers of Borel diffused probability measures on ${\bf R}$
}
\author{ M.Kintsurashvili, T.Kiria and  G.Pantsulaia }%
\address{Department of Mathematics, Georgian Technical University,\\
Kostava Street-77,  Tbilisi 0175, Georgia}%
\email{  m.kintsurashvili@gtu.ge, t.kiria@gtu.ge, g.pantsulaia@gtu.ge }%
\thanks{The third author partially is
supported by Shota Rustaveli National Science Foundation's Grant
no 31/25.}%

\begin{abstract}The paper contains a brief description of Yamasaki's remarkable investigation (1980) of the relationship  between Moore-Yamasaki-Kharazishvili type measures and  infinite powers of Borel diffused probability measures on ${\bf R}$. More precisely, we give   Yamasaki's proof that no infinite power of the Borel probability measure with a strictly positive density function on $R$ has an equivalent Moore-Yamasaki-Kharazishvili type measure. A certain modification of Yamasaki's example is used for the construction of such a Moore-Yamasaki-Kharazishvili type measure  that  is equivalent to the  product of a certain infinite family of Borel probability measures with a strictly positive density function on $R$.  By virtue of the properties of  equidistributed sequences on the real axis,  it is demonstrated that an arbitrary  family of infinite powers of Borel diffused probability measures with strictly positive density functions on $R$ is strongly separated and, accordingly, has an infinite-sample well-founded estimator of the unknown distribution function. This extends the main result  established  in [\textit{ Zerakidze Z., Pantsulaia G.,  Saatashvili G.}  On the separation
problem for a family of Borel and Baire $G$-powers of
shift-measures on $\mathbb{R}$ // Ukrainian Math. J. -2013.- {\textbf 65 (4)}.- P. 470--485 ].
\end{abstract}
\maketitle
\section{Introduction}

Let $\mu$  and $\nu$ be non-trivial $\sigma$-finite  measures on a measurable space $(X,M)$. The measures $\mu$  and $\nu$ are called orthogonal if there is a measurable set
 $E \in M$ such that $\mu(E)=0$ and $\nu(X \setminus E)=0$.  The measures $\mu$  and $\nu$ are called equivalent if and only if the following condition
 $$
 (\forall E)(E \in M \rightarrow (\mu(E)=0 \Longleftrightarrow \nu(E)=0))
 $$
 is satisfied.

 It is well known that the following  facts hold true in an $n$-dimensional Euclidean vector space ${\bf R}^n~(n \in N)$:

{\bf Fact 1.1}~{\it  Let $\mu$ be a probability Borel measure on ${\bf R}$ with a strictly positive
continuous distribution function  and  $\lambda_n$ be a Lebesgue measure defined on the
$n$-dimensional topological vector space ${\bf R}^n$. Then the measures $\mu^n$ and $\lambda_n$ are equivalent.}

\medskip

{\bf Fact 1.2}~{\it Let $(\mu_k)_{1 \le k \le n}$ be a family of  Borel probability measures on ${\bf R}$ with strictly positive
continuous distribution functions  and $\lambda_n$ be a Lebesgue measure defined on the
$n$-dimensional topological vector space $\mathbf{R}^n$. Then the measures $\prod_{k=1}^n\mu_k$ and $\lambda_n$ are equivalent.}

\medskip

{\bf Fact 1.3}~{\it  Let $\mu_1$  and $\mu_2$ be Borel probability measures on ${\bf R}$ with strictly positive
continuous distribution functions. Then the measures $\mu_1^n$ and $\mu_2^n$ are equivalent.}

\medskip

{\bf Fact 1.4}~{\it Let $(\mu_k)_{1 \le k \le n}$ be a family of  Borel probability measures on ${\bf R}$ with strictly positive
continuous distribution functions. Then the measures $\mu_k^n$ and $\mu_l^n$ are equivalent for each $1 \le k \le l \le n$.}
\medskip

The proof of the above mentioned facts employs  the  following simple lemma which is well known  in the literature.

{\bf Lemma 1.1}~{\it  Let $\mu_k$  and $\nu_k$ be equivalent non-trivial $\sigma$-finite Borel measures on the measurable space $(X_k,M_k)$  for $1 \le k \le n$. Then
the measures $\prod_{k=1}^n\mu_k$ and $\prod_{k=1}^n\nu_k$  are equivalent.}

\begin{proof} Note that for proving  Lemma 1.1, it  suffices  to prove that  if $\mu_k$  is absolutely continuous with respect to $\nu_k(k=1,2)$, then
  so is $\prod_{k=1}^2\mu_k$   with respect to $\prod_{k=1}^2\nu_k$.

Assume that $E \in M_1 \times M_2$ such that $\mu_1\times \mu_2(E)=0$. We have to show that $\nu_1\times \nu_2(E)=0$.

By the Fubini theorem we have
$$0=\mu_1\times \mu_2(E)=\int_{X_1}\mu_2(E \cap (\{x\} \times X_2))d\mu_1(x).$$
This means that
$$
\mu_{1}(\{ x : \mu_2(E \cap (\{x\} \times X_2))>0 \})=0
$$
or, equivalently,
$$
\mu_{1}( X_1  \setminus  \{ x : \mu_2(E \cap (\{x\} \times X_2))=0 \})=0.
$$
Since $\nu_1 \ll \mu_1,$ we have
$$
 \{ x : \mu_2(E \cap (\{x\} \times X_2))=0 \} \subseteq \{ x : \nu_2(E \cap (\{x\} \times X_2))=0 \}.
$$
Since $\nu_1 \ll \mu_1$ and
$$
\mu_{1}( X_1  \setminus  \{ x : \nu_2(E \cap (\{x\} \times X_2))=0\})=0,
$$
we have
$$
\nu_{1}( X_1  \setminus  \{ x : \nu_2(E \cap (\{x\} \times X_2))=0\})=\nu_{1}(\{ x : \nu_2(E\cap (\{x\} \times X_2))>0\})=0.
$$
Finally, we get
$$\nu_1\times \nu_2(E)=\int_{X_1}\nu_2(E \cap (\{x\} \times X_2))d\nu_1(x)=
$$
$$
\int_{\{ x : \nu_2(E \cap (\{x\} \times X_2))>0\} }\nu_2(E \cap (\{x\} \times X_2))d\nu_1(x)+
$$
$$
\int_{\{ x : \nu_2(E \cap (\{x\} \times X_2))=0\}}\nu_2(E \cap (\{x\} \times X_2))d\nu_1(x)=0.
$$
\end{proof}

In order to obtain the infinite-dimensional versions of Facts 1.1-1.2, we must  know  what measures in infinite-dimensional topological vector
spaces can be taken as  partial analogs of the Lebesgue measure in ${\bf R}^n~(n \in N)$. In this direction  the results of I. Girsanov and B. Mityagyn \cite{Gir59} and Sudakov \cite{Sud59}
on the nonexistence of nontrivial translation-invariant
$\sigma$-finite Borel measures in infinite-dimensional topological
vector spaces are important. These authors assert that the
properties of $\sigma$-finiteness and of translation-invariance
are not consistent. Hence one can weaken the property of
translation-invariance for analogs of the Lebesgue measure and
construct nontrivial $\sigma$-finite Borel measures
which are invariant under everywhere dense linear manifolds. We wish to make a special note that Moore \cite{Moore65}, Yamasaki \cite{Yam80} and Kharazishvili \cite{Khar84} give the constructions of such measures
 in an infinite-dimensional Polish topological vector space $\mathbb{R^N}$
of all real-valued sequences equipped with product topology,  which are invariant under the group $R^{(N)}$ of all eventually zero real-valued
sequences. Such measures can be called Moore-Yamasaki-Kharazishvili  type measures in $\mathbb{R^N}$.  Using Kharazishvili's approach \cite{Khar84},
it is  proved in \cite{GKPP14} that every infinite-dimensional Polish linear space admits a $\sigma$-finite
non-trivial Borel measure that is translation invariant with respect to a dense
linear subspace. This extends a recent result of Gill, Pantsulaia and Zachary \cite{GPZ12}
on the existence of such measures in Banach spaces with Schauder bases.

In this paper, we focus on the question whether Facts 1.1-1.2  admit infinite-dimensional generalizations in terms of Moore-Yamasaki-Kharazishvili  type measures in $\mathbb{R^N}$. To this end, our consideration will involve the following problems.

\medskip

{\bf Problem 1.1}~{\it  Let $\mu$ be  a probability Borel measure on $R$ with a strictly positive
continuous distribution function  and  $\lambda$ be a Moore-Yamasaki-Kharazishvili  type measure in $\mathbb{R^N}$. Are the measures $\mu^N$ and $\lambda$ equivalent ?}

\medskip

{\bf Problem 1.2}~{\it Let $(\mu_k)_{k \in N}$ be a family of  Borel probability measures on $R$ with strictly positive
continuous distribution functions  and  $\lambda$ be a Moore-Yamasaki-Kharazishvili  type measure in $\mathbb{R^N}$. Are the measures $\prod_{k \in N}\mu_k$ and $\lambda$  equivalent?}

\medskip

Concerning Facts 1.3-1.4, it is natural to consider the following problems.

\medskip

{\bf Problem 1.3}~{\it Let $\mu_1$ and $\mu_2$ be Borel probability measures on $R$ with strictly positive
continuous distribution functions.  Are the measures $\mu_1^N$ and $\mu_2^N$ equivalent?}

\medskip

{\bf Problem 1.4}~{\it  Let $(\mu_i)_{i \in I} $ be a family of all Borel probability measures on $R$ with strictly positive
continuous distribution functions. Setting $S(R^N):=\cap_{i \in I}\mbox{dom}(\overline{\mu_i^N}),$ where $\overline{\mu_i^N}$ denotes a usual completion of the measure $\mu_i^N(i \in I)$, does there exist a partition $(D_i)_{i \in I}$ of $R^N$ into elements of the $\sigma$-algebra $S(R^N)$ such that $\overline{\mu_i^N}(D_i)=1$ for each $i \in I$?}

\medskip

Problems 1.3-1.4 are not new and  have
been investigated by many authors in more general formulations. In this direction, we should specially  mention
the result of S. Kakutani \cite{Kak48}(see
Theorem 4.3) stating that if one has equivalent probability measures
$\mu_i$ and $\nu_i$ on the $\sigma$-algebra $\mathcal{L}_i$ of
subsets of a set $\Omega_i, i =1,2,\cdots$ and if $\mu$ and $\nu$
denote respectively the infinite product measures $\prod_{i \in
N}{\mu}_i$ and $\prod_{i \in N}{\nu}_i$ on the infinite product
$\sigma$-algebra  generated on the infinite product set $\Omega$,
then $\mu$ and $\nu$ are either equivalent or orthogonal. Similar dichotomies have revealed themselves in
the study of Gaussian stochastic processes. C. Cameron and
W.E. Martin \cite{CamMar44} that if one considers the measures
induced on a path space by a Wiener process on the unit interval,
then, if the variances of corresponding processes are different,  the measures
are orthogonal. Results of this kind were generalized by many
authors(cf. \cite{Feld58}, \cite{Gren50} and others). A.M. Vershik \cite{Ver66}
proved that a group of all admissible translations(in the sense of
quasiinvariance) of an arbitrary Gaussian measure in an
infinite-dimensional separable Hilbert space is a linear manifold.

For study of the general problem of equivalence and
singularity  of two product measures was carried out by  various authors using different approaches, among which are the strong law of large numbers,  the properties of the Hellinger integral \cite{Hill71}, the zero-one laws  \cite{Lep72} and so on. In this paper,  we propose  a new approach for the solution of  Problems 1.3-1.4,   which uses the properties of uniformly distributed sequences \cite{KuNi74}.

In Sections 2-3,  we give solutions of Problems 1.1-1.2 which are due to  Yamasaki \cite{Yam80}. In Section 4,  we give solutions of Problems 1.3-1.4.

\section{Negative solution of the Problem 1.1}

A negative solution of  Problem 1.1  is contained in the following

\medskip

{\bf Fact 2.1}(\cite{Yam80}, Proposition 2.1, p. 696){\it  Let $f(x)$ be a measurable function on $R^1$ which satisfies
$f(x)>0$  and $\int_{-\infty}^{+\infty} f(x)dx = 1$. Let $\mu$ be the stationary product measure of $f$ (i.e.
$d\mu= \prod_{i=1}^{\infty}f(x_i)dx_i$) and ${\bf R^{(N)}}$ be a linear vector space of all eventually zero real-valued sequences. Then $\mu$ is ${\bf R^{(N)}}$-quasi-invariant but $\mu$  has no equivalent Moore-Yamasaki-Kharazishvili type measure.}

\begin{proof} As proved in \cite{Shim75}, the stationary product measure
$\mu$ is ${\bf R^{(N)}}$-ergodic. Let $\sum$ be the permutation group on the set of all natural numbers
$N = \{1, 2,...\}.$  $\sum$ can be regarded as a transformation group on ${\bf R^N}$, and $\mu$
is $\sum$ -invariant. Let $\sum_0$ be the subgroup of $\sum$ generated by all transpositions (of
two elements of $N$).  $\sum_0$ consists of such a permutation $\sigma \in \sum$ that satisfies $\sigma(i) = i$
except finite numbers of $i \in N$. As shown in \cite{Shim75}, the measure $\mu$ is $\sum_0$-ergodic.

Now, we shall derive a contradiction assuming that $\mu$ has an equivalent
${\bf R^{(N)}}$-invariant $\sigma$-finite measure $\nu$. Since $\mu \approx \nu$, where  $\mu$  is $\sum_0$-invariant and $\sum_0$-
ergodic, and $\nu$ is  $\sum_0$-invariant, then we have $\mu=c \nu$ for some constant $c>0$. Thus,
the ${\bf R^{(N)}}$-invariance of $\nu$ implies that of $\mu$, which is a contradiction.

Therefore  it suffices  to prove that $\nu$ is $\sum_0$-invariant, namely for each $\sigma \in \sum_0$, $\tau_{\sigma}\nu=\nu$, where
$$
\tau_{\sigma}\nu(B)= \nu(\sigma^{-1}(B)),  \eqno (2.1)
$$
for each $B \in \mathbb{B}(R^{N})$.
Since $\tau_{\sigma}\mu=\mu$, we have $\tau_{\sigma}\nu \approx \nu$. On the other hand, $\nu$ is ${\bf R^{(N)}}$-ergodic because $\mu$
is such. Therefore if $\tau_{\sigma}\nu$ is ${\bf R^{(N)}}$-invariant, then we have $\tau_{\sigma}\nu =c_{\sigma}\nu$ for some constant
$c_{\sigma}>0$. In particular for a transposition $\sigma$, ${\sigma}^2=l$ implies $c_{\sigma}^2=l$, hence $c_{\sigma}=1$.
This means that $\nu$ is invariant under any transposition. Since $\sum_0$  is generated
by the set of all transpositions, we have proved the $\sum_0$-invariance of $\nu$.

To complete the proof of Fact 2.1, it remains only to prove that $\tau_{\sigma}\nu$ is
${\bf R^{(N)}}$-invariant. Since $\nu$  is ${\bf R^{(N)}}$ -invariant, we have $\tau_{x}\nu=\nu$ for any $x \in {\bf R^{(N)}}$. Therefore
$$
(\forall x)(x \in {\bf R^{(N)}} \rightarrow \tau_{\sigma}\tau_{x}\nu=\tau_{\sigma}\nu). \eqno (2.2)
$$
However, we can easily show $\tau_{\sigma}\tau_{x}\nu=\tau_{\sigma x}\tau_{\sigma}\nu,$ so (2.2) implies that $\tau_{\sigma}$ is $\sigma({\bf R^{(N)}})$-invariant.
Since $\sigma$  maps ${\bf R^{(N)}}$ onto ${\bf R^{(N)}}$, namely $\sigma({\bf R^{(N)}})= {\bf R^{(N)}}$
we have proved the
${\bf R^{(N)}}$-invariance of $\tau_{\sigma}\nu$.

\end{proof}

\section{Particular  solution of  Problem 1.2}

{\bf Remark 3.1.} If in the formulation of Problem 1.2 we have that  $\mu_k=\mu_n$ for each $k,n \in N$, then,  following Fact 2.1, the answer to the question posed in Problem 2.1 is no.

\medskip

{\bf Example 3.1.}(\cite{Pan02}, Section 1, p. 354 ). Let $ {\bf R}^N $ be the topological vector  space of all real-valued
sequences equipped with the Tychonoff topology. Let us denote by $B({\bf R}^N) $ the $\sigma$-algebra of all Borel subsets in $ {\bf
R}^N $.

  Let $ (a_i)_{i \in N} $ and $ (b_i)_{i \in N} $ be sequences of real numbers
such that
$$
( \forall i )( i \in N \rightarrow a_i < b_i ).
$$

We put
$$
 A_n={\bf R}_0 \times \cdots \times {\bf R}_n \times (\prod \limits_{i > n}\Delta_i)~,
$$
for $n \in N$, where
$$
(\forall i)( i \in N \rightarrow {\bf R}_i={\bf R}~ \& ~
\Delta_i=[a_i;b_i[).
$$
We put also
$$
\Delta=\prod_{i \in N}\Delta_i.
$$

   For an arbitrary natural number $i \in N$, consider the Lebesgue measure
$ \mu_i $  defined on the space ${\bf R}_i$ and satisfying the
condition $\mu_i(\Delta_i)=1$. Let us denote by $\lambda_i$ the
normed Lebesgue measure defined on the interval $\Delta_i$.

   For an  arbitrary $n \in N$, let us denote by $\nu_n$ the measure  defined by
$$
\nu_n= \prod \limits_{1 \le i \le n} \mu_i \times \prod\limits_{i
> n} \lambda_i,
$$
and by ${\overline{\nu}}_n$ the Borel measure in the space ${\bf
R}^N$  defined by
$$
 ( \forall X)(X \in B({\bf R}^N)  \rightarrow  {\overline{\nu}}_n(X)=
\nu_n(X \cap A_n)).
$$

   Note that (see  \cite{Pan02}, Lemma 1.1, p. 354) for an arbitrary Borel set $X \subseteq {\bf
R}^N $ there exists a limit
$$
{\nu}_{\Delta}(X)= \lim \limits_{n \rightarrow \infty}
\overline{\nu}_n(X).
$$
Moreover, the functional ${\nu}_{\Delta}$ is a nontrivial
$\sigma$-finite measure  defined on the Borel $\sigma$-algebra
$B({\bf R}^N)$.

\medskip

Recall that an element $h \in {\bf R}^N$ is called an admissible
translation in the sense of invariance  for the measure
${\nu}_{\Delta}$ if
$$
(\forall X)(X \in B({\bf R}^N) \rightarrow {\nu}_{\Delta}(X+h)=
{\nu}_{\Delta}(X)).
$$

We define
$$
G_{\Delta}=\{ h: h \in {\bf R}^N ~  \& ~ h~ \mbox{is an admissible
translation for} ~{\nu}_{\Delta} \}.
$$

  It is easy to show that $G_{\Delta}$ is a vector subspace of ${\bf R}^N$.

We have the following

\begin{lem} (\cite{Pan02},Theorem 1.4, p.356) The following conditions are equivalent:
$$
1) ~~ g=(g_1, g_2,\cdots) \in G_{\Delta},
$$
$$
2) ~(\exists n_g)(n_g \in N  \rightarrow ~\mbox{the series} \sum_{i \ge n_g}\ln( 1- \frac{|g_i|}{b_i-a_i})~
\mbox{is convergent}).
$$
\end{lem}

 Let ${\bf R}^{(N)}$ be the space of all finite
sequences, i.e.,
$$
{\bf R}^{(N)}=\{(g_i)_{i \in N}|(g_i)_{i \in N} \in {\bf R}^N ~\&~
\mbox{card}\{i|g_i\neq 0\} < \aleph_0 \}.
$$

  It is clear that, on the one hand, for an arbitrary compact infinite-dimensional
parallelepiped $\Delta=\prod\limits_{k \in N}[a_k,b_k]$, we have
$$
{\bf R}^{(N)} \subset G_{\Delta}.
$$
 On the other hand, $G_{\Delta} \setminus {\bf R}^{(N)} \neq \emptyset$ since an
element $(g_i)_{i \in N}$  defined by
$$
(\forall i)(i \in N \rightarrow g_i=(1- \mbox{exp} \{ -
\frac{b_i-a_i}{2^i} \} \times (b_i-a_i)))
$$
belongs to the difference $G_{\Delta} \setminus {\bf R}^{(N)}.$

  It is easy to show  that the vector space $G_{\Delta}$ is everywhere dense
in ${\bf R}^N$ with respect to the Tychonoff topology since ${\bf
R}^{(N)} \subset G_{\Delta}$.

\medskip

Below we present  an example of the product of an infinite family of  Borel probability measures on $R$ with strictly positive
continuous distribution functions  and a Moore-Yamasaki-Kharazishvili  type measure in $\mathbb{R^N}$, such that these  measures are equivalent.

\medskip

Let $(c_n)_{n \in N}$ be a sequence of positive numbers such that $0<c_n< l $. On the
real axis  $R$, for each $n$ consider a continuous function $f_n(x)$ which satisfies:
$$
0< f_n(x)<1, \int_{-\infty}^{+\infty}f_n(x)dx=1,
$$
$$
f_n(x)=c_k ~\mbox{for}~ x  \in [0,1] .
$$

Such a function $f_n(x)$ exists certainly  for any $n \in N$.

For $n \in N$, let us denote by  $\mu_n$ a Borel probability measure on $R$ defined  by the distribution density function $f_n$.

\medskip

{\bf Fact 3.1}~{\it  If $\prod_{n \in N}c_n >0$, then the measures $\prod_{n \in N}\mu_n$ and  $\nu_{[0,1]^{N}}$ are equivalent.}

\begin{proof} By the Fubini theorem, one can easily prove that the measure $\prod_{n \in N}\mu_n$ is ${\bf R}^{(N)}$-quasiinvariant.
According to \cite{Shim75}, every
product measure on ${\bf R}^{N}$  is ${\bf R}^{(N)}$-ergodic. Therefore, $\prod_{n \in N}\mu_n$, hence $\nu_{[0,1]^{N}}$,  too is ${\bf R}^{(N)}$-
ergodic.

For $x = (x_n) \in {\bf R}^{N}$, define a function  $f(x)$ by:
$$
f(x)=\prod_{n \in N}f_n(x).  \eqno (3.1)
$$
Since $0<f(x_n)<1$, the partial product decreases monotonically, so that the
infinite product in (3.1) exists certainly. If $x \in A_n$, then $x_k \in [0,1]$ for $k>n$, so
we have
$$
f(x)=\prod_{k=1}^n f_k(x_k) \prod_{k>n} c_k > 0.
$$
Thus $f(x)$ is positive on $A_n$, hence positive on $\cup_{n \in N}A_n$, too.  On the other hand, since
$\nu_{[0,1]^{N}}({\bf R}^{N} \setminus \cup_{n \in N}A_n)=0,$  we see that $f(x)$  is positive for $\nu_{[0,1]^{N}}$-almost all $x$.

Now, define a measure $\nu^{'}$  on ${\bf R}^{N}$ by
$$
\nu^{'}(X)=\int_{X}f(x)d \nu_{[0,1]^{N}}(x)
$$
for $X \in \mathbf{B}({\bf R}^{N})$.

Let us show that $\prod_{n \in N}\mu_n=\nu^{'}$. For this it suffices to show that for each $A \in \mathbf{B}(R^n)$ we have
$$
\nu^{'}(A \times {\bf R}^{N \setminus \{ 1, \cdots, n\}})=\prod_{n \in N}\mu_n(A \times {\bf R}^{N \setminus \{ 1, \cdots, n\}}).
$$
Indeed, we have

$$
\nu^{'}(A \times {\bf R}^{N \setminus \{ 1, \cdots, n\}})=\int_{A \times {\bf R}^{N \setminus \{ 1, \cdots, n\}}}f(x)d \nu_{[0,1]^{N}}(x)=
$$
$$
\lim_{m \to +\infty}\int_{A_m \cap (A \times {\bf R}^{N \setminus \{ 1, \cdots, n\}})}f(x)d \nu_{[0,1]^{N}}(x)=
$$
$$
\lim_{m \to +\infty}\int_{A \times \prod_{k=n+1}^m {\bf R} \times \prod_{k >m}[0,1]}f(x)d \nu_{[0,1]^{N}}(x)=
$$
$$
\lim_{m \to +\infty}\int_{A \times \prod_{k=n+1}^m {\bf R} \times \prod_{k >m}[0,1]}f(x)d (\prod_{k=1}^m \mu_k \times \prod_{k>m}\lambda_k)=
$$
$$
\lim_{m \to +\infty} \int_{A \times \prod_{k=n+1}^m {\bf R}}(\int_{\prod_{k >m}[0,1]}f(x)d\prod_{k>m}\lambda_k)d \prod_{k=1}^m\mu_k=
$$
$$
\lim_{m \to +\infty}\int_{\prod_{k >m}[0,1]}\prod_{k>m}f_k(x_k)d\prod_{k>m}\lambda_k \times
 $$
 $$\lim_{m \to +\infty} \int_{A \times \prod_{k=n+1}^m {\bf R}}\prod_{k=1}^mf_k(x_k)d \prod_{k=1}^m\mu_k=
$$

$$
\lim_{m \to +\infty}\int_{\prod_{k >m}[0,1]}\prod_{k>m}f_k(x_k)d\prod_{k>m}\lambda_k \times
 $$
 $$\lim_{m \to +\infty} \int_{A}\prod_{k=1}^n f_k(x_k)d \prod_{k=1}^n\mu_k \times \int_{\prod_{k=n+1}^m {\bf R}}\prod_{k=n+1}^mf_k(x_k)d \prod_{k=n+1}^m\mu_k =
$$

$$
\lim_{m \to +\infty} \prod_{k>m}c_k \times
\prod_{k=1}^n\mu_k(A)=\prod_{k=1}^n\mu_k(A)=\prod_{k \in N}\mu_k(A \times {\bf R}^{N \setminus \{ 1, \cdots, n\}}  ) .
$$
This ends the proof of Fact 3.1.

\end{proof}

\begin{rem} Let the product-measure $\prod_{k \in N}\mu_k$ comes from  Fact 3.1. Then by virtue of Lemma 3.1, we know
 that the group of all admissible translations (in the sense of invariance) for the measure
 $\nu_{[0,1]^{N}}$ is $l_1=\{ (x_k)_{k \in N}: (x_k)_{k \in N}\in R^N~\&~\sum_{k \in N}|x_k|<+\infty\}$.
  Following Fact 3.1, the measures  $\prod_{k \in N}\mu_k$ and $\nu_{[0,1]^{N}}$ are equivalent,  which implies that
  the group of all admissible translations (in the sense of quasiinvariance) for the measure $\prod_{k \in N}\mu_k$ is equal to $l_1$.

For $(x_k)_{k \in N} \in l_1$,  we set $\nu_k(X)=\mu_k(X-x_k)$ for each $X \in B(R)$. It is obvious that $\mu_k$ and~ $\nu_k$ are equivalent for each $k \in N$. For $k \in N$ and $x \in R$, we put
$\rho_k(x)=\frac{d\nu_k(x)}{d\mu_k(x)}$. Let us consider the product-measures $\mu=\prod\limits_{k \in N}\mu_k$ and
$\nu=\prod\limits_{k \in N}\nu_k$. On the one hand, following our observation,  the measures $\mu$ and $\nu$ are equivalent.
On the other hand, by virtue  of Kakutani's well known  result (see, \cite{Kak48}), since the measures $\mu$ and $\nu$ are
equivalent, we deduce that the infinite product $\prod\limits_{k \in
N}\alpha_k$~ is divergent to zero, where
$\alpha_k=\int\limits_{R}\sqrt{\rho_k(x_k)}d\mu_k(x_k)$. In this
case $r_n(x)=\prod\limits_{k=1}^n\rho_k(x)$~is convergent (in the
mean) to the function $r(x)=\prod\limits_{k=1}^{\infty}\rho_k(x)$
which is the density of the measure $\nu$  with respect to $\mu$,
i.e.,
$$
r(x)=\frac{d\nu(x)}{d\mu(x)}.
$$
\end{rem}

\begin{rem} The approach used in the proof of Fact 3.1 is taken  from  \cite{Yam80}(see  Proposition 4.1, p. 702).

\end{rem}

In the context of Fact 3.1 we state the following

\medskip

{\bf Problem 3.1} Do there exist a family $(\mu_k)_{k \in N}$ of linear Gaussian probability measures on $R$ and a Moore-Yamasaki-Kharazishvili type measure $\lambda$ in  ${\bf R^N}$ such that the measures $\prod\limits_{k \in N}\mu_k$ and $\lambda$ are equivalent?

\section{Solution of Problems 1.3 - 1.4}

We present a new approach for the solution of Problems 1.3 - 1.4, which is quite different from the approach introduced in \cite{Kak48}. Our approach uses the technique of the so-called uniformly distributed sequences. The main notions and auxiliary propositions are taken  from \cite{KuNi74}.

\begin{defn} \cite{KuNi74}
~ A sequence $(x_k)_{k \in N}$ of real numbers from the interval
$(a, b)$ is said to be equidistributed or uniformly distributed on
an interval $(a, b)$ if for any subinterval $[c, d]$ of  $(a, b)$
we have
$$ \lim_{n \to \infty}
n^{-1}\#(\{x_1, x_2, \cdots, x_n\} \cap [c,d])=(b-a)^{-1}(d-c),
$$ where $\#$ denotes the counting measure.
\end{defn}

Now let $X$ be a compact Polish space and $\mu$ be a probability
Borel measure on $X$. Let $\mathcal{R}(X)$ be a space of all
bounded continuous measurable functions defined on  $X$.

\begin{defn}  A sequence
$(x_k)_{k \in N}$ of elements of $X$ is said to be
$\mu$-equidistributed or $\mu$-uniformly distributed on $X$
 if for every $f \in \mathcal{R}(X)$
we have
$$
\lim_{n \to \infty}n^{-1}\sum_{k=1}^nf(x_k) =\int_{X}fd\mu.
$$

\end{defn}
\begin{defn} (\cite{KuNi74}, Lemma 2.1, p.
199)~Let
 $f \in \mathcal{R}(X)$. Then, for  $\mu^N$-almost
 every sequence $(x_k)_{k \in N} \in X^{N}$,   we have
$$
\lim_{n \to \infty}n^{-1}\sum_{k=1}^nf(x_k)=\int_{X}fd\mu.
$$
\end{defn}

\begin{lem}(\cite{KuNi74}, pp. 199-201)~
Let $S$ be a set of all $\mu$-equidistributed sequences on $X$.
Then we have $\mu^{N}(S)=1$.
\end{lem}

\begin{cor}(\cite{ZerPanSaat13}, Corollary 2.3, p.
473)  Let $\ell_1$ be a Lebesgue measure on $(0,1)$.
 Let $D$ be a set of all $\ell_1$-equidistributed sequences on
$(0,1)$. Then we have ~$\ell_1^{N}(D)=1$.
\end{cor}

\begin{defn}
Let $\mu$ be a
probability Borel measure on $R$ with a distribution function $F$. A sequence $(x_k)_{k \in N}$ of elements of
$R$ is said to be $\mu$-equidistributed or $\mu$-uniformly
distributed on $R$ if for every interval $[a, b] (-\infty \le a <
b \le +\infty)$ we have
$$
\lim_{n \to \infty} n^{-1} \# ([a,b] \cap \{x_1, \cdots, x_n \}
)=F(b)-F(a).
$$
\end{defn}

\begin{lem}(\cite{ZerPanSaat13}, Lemma 2.4, p.
473) Let $(x_k)_{k \in N}$ be an
$\ell_1$-equidistributed sequence on $(0,1)$, $F$ be a strictly
increasing  continuous distribution function on $R$ and $p$ be a
Borel probability measure on $R$ defined by $F$. Then
$(F^{-1}(x_k))_{k \in N}$ is $p$-equidistributed on $R$.
\end{lem}

\begin{cor}(\cite{ZerPanSaat13}, Corollary 2.4, p.
473)  Let $F$ be a strictly
increasing continuous distribution  function on $R$ and $p$ be a
Borel probability measure on $R$ defined by $F$. Then for a set
$D_{F} \subset R^{N}$ of all $p$-equidistributed sequences  on $R$
we have :

(i) $D_{F}= \{ (F^{-1}(x_k))_{k \in N} : (x_k)_{k \in N} \in D
\}$;

(ii)~$p^N(D_{F})=1$.
\end{cor}
\medskip
\begin{lem}
 Let $F_1$  and $F_2$ be different strictly
increasing continuous distribution  functions on $R$,  and $p_1$ and $p_2$ be
Borel probability measures on $R$ defined by $F_1$ and $F_2$, respectively.  Then there does not exist
a sequence of real numbers $(x_k)_{k \in \mathbf{N}}$ which simultaneously is $p_1$-equidistributed and $p_2$-equidistributed.
\end{lem}
\begin{proof}
Assume the contrary and let $(x_k)_{k \in \mathbf{N}}$ be such a sequence. Since
$F_1$  and $F_2$  are different,  there is a point $x_0 \in \mathbb{R}$ such that $F_1(x_0)\neq F_2(x_0)$. The latter relation is not possible under our assumption  because
$(x_k)_{k \in \mathbf{N}}$ simultaneously is $p_1$-equidistributed and $p_2$-equidistributed, which implies
$$
F_1(x_0)=\lim_{n \to \infty}n^{-1}\# ((-\infty,x_0] \cap \{x_1, \cdots,
x_n \} )=F_2(x_0).
$$

\end{proof}

The next theorem contains the solution of Problem 1.3.

\begin{thm}
 Let $F_1$  and $F_2$ be different strictly
increasing continuous distribution  functions on $R$ and $p_1$ and $p_2$ be
Borel probability measures on $R$,  defined by $F_1$ and $F_2$, respectively.  Then the measures $p^N_1$ and $p^N_2$  are orthogonal.
\end{thm}
\begin{proof} Let $D_{F_1}$ and $D_{F_2}$  denote $p_1$-equidistributed and $p_2$-equidistributed sequences on $R$, respectively.  By Lemma 4.9  we know that
$D_{F_1} \cap D_{F_2}=\emptyset.$  By Corollary 4.8 we know that $p^N_1(D_{F_1})=1$ and $p^N_2(D_{F_2})=1$. This ends the proof of the theorem.
\end{proof}

\begin{defn} Let  $\{ \mu_i : i \in I\}$ be a  family  of probability measures defined on a measure space $(X,M)$. Let $S(X)$ be defined by
$$ S(X)=\cap_{i \in I}\mbox{dom}(\overline{\mu}_i),$$
where $\overline{\mu}_i$ denotes a usual completion of the measure $\mu_i$. We say that the family $\{ \mu_i : i \in I\}$ is strongly separable if  there  exists a partition $\{C_i :i  \in I\}$ of the space $X$ into elements of the $\sigma$-algebra $S(X)$ such that $\overline{\mu}_i(C_i)=1$ for each $i \in I$.
\end{defn}

\begin{defn} Let  $\{ \mu_i : i \in I\}$ be a  family  of probability measures defined on a measure space $(X,M)$. Let $S(I)$ denote a minimal $\sigma$-algebra generated by singletons of $I$  and the $\sigma$-algebra $S(X)$ of subsets of  $X$ be defined by
$$ S(X)=\cap_{i \in I}\mbox{dom}(\overline{\mu}_i),$$
where $\overline{\mu}_i$ denotes a usual completion of the measure $\mu_i$ for $i \in I$.
We say that a $(S(X),S(I))$-measurable mapping $T:X \to I$ is a well-founded estimate of an unknown parameter $i ~(i \in I)$ for the  family $\{ \mu_i : i \in I\}$  if the following condition
$$
(\forall i)(i \in I \rightarrow \mu_i(T^{-1}(\{i\})=1))
$$
holds true.
\end{defn}

One can easily get the validity of the following assertion.

\begin{lem}(\cite{ZerPanSaat13}, Lemma 2.5, p.
474) Let  $\{ \mu_i : i \in I\}$ be a  family  of probability measures defined on a measure space $(X,M)$.  The following propositions are equivalent:

(i)~ The family  of probability measures $\{ \mu_i : i \in I\}$ is strongly separable;

(ii) There exists  a well-founded estimate of an unknown  parameter $i ~(i \in I)$ for the  family $\{ \mu_i : i \in I\}$.

\end{lem}

The next theorem contains the solution of Problem 1.4.

\begin{thm} Let  $\mathcal{F}$ be a  family  of all strictly increasing and continuous distribution functions on ${\bf R}$  and $p_{F}$ be a Borel probability measure
on $R$ defined by $F$ for each $F \in \mathcal{F}$. Then the family of Borel probability measures  $\{ p_F^N : F \in \mathcal{F})\} $ is strongly separable.
\end{thm}
\begin{proof} We denote by $D_{F}$ the set of all  $p_F$-equidistributed  sequences on $R$  for each $F \in \mathcal{F}$.  By Lemma 4.9  we know that
$D_{F_1} \cap D_{F_2}=\emptyset$ for each different $F_1,F_2 \in \mathcal{F}$.  By Corollary 4.8 we know that $p^N_F(D_{F})=1$    for each $F \in \mathcal{F}$.  Let us fix  $F_0 \in \mathcal{F}$ and define a family $(C_F)_{F  \in \mathcal{F}}$ of subsets of ${\bf R^N}$ as follows: $C_F=D_F$ for $F \in \mathcal{F} \setminus \{F_0\}$ and $C_{F_0}=R^N \setminus \cup_{ F \in \mathcal{F} \setminus \{F_0\}}D_F$. Since $D_F$ is a Borel subset of ${\bf R^N}$ for each $F \in \mathcal{F}$,  we claim that $C_F \in S({\bf R^N})$ for each $F \in \mathcal{F} \setminus \{F_0\}$. Since
$\overline{p_F^N}(R^N \setminus \cup_{ F \in \mathcal{F}}D_F)=0$ for each  $F \in \mathcal{F}$, we deduce that $R^N \setminus \cup_{ F \in \mathcal{F}}D_F \in \cap_{F \in \mathcal{F}}\mbox{dom}(\overline{p_F^N})=S(R^N)$. Since $S({\bf R^N})$ is an $\sigma$-algebra, we claim that $C_{F_0} \in S(R^N)$  because
$\overline{p_F^N}(R^N \setminus \cup_{ F \in \mathcal{F}}D_F)=0$ for each  $F \in \mathcal{F}$(equivalently, $R^N \setminus \cup_{ F \in \mathcal{F}}D_F  \in S(R^N)$), and

$$C_{F_0}=R^N \setminus \cup_{ F \in \mathcal{F} \setminus \{F_0\}}D_F=(R^N \setminus \cup_{ F \in \mathcal{F}}D_F)\cup D_{F_0}.
$$
This ends the proof of the theorem.

\end{proof}

By virtue of the results of Lemma 4.13 and Theorem 4.14 we get the following

\begin{cor} Let  $\mathcal{F}$ be a  family  of all strictly increasing and continuous distribution functions on ${\bf R}$. Then there exists  a well-founded estimate of an unknown distribution function  $F  ~(F \in \mathcal{F})$ for the family of Borel probability measures  $\{ p_F^N : F \in \mathcal{F}\}$.
\end{cor}

\begin{rem}The validity of Theorem 4.14 and Corollary 4.15 can be obtained for an arbitrary family of strictly increasing and continuous distribution functions on ${\bf R}$.
Note that Corollary 4.15 extends the main result established  in  \cite{ZerPanSaat13}(see  Lemma 2.6, p. 476).
\end{rem}

\begin{rem} The requirements in Theorem 4.14  that all  Borel probability measures on  ${\bf R}$ are defined by  strictly increasing and continuous distribution functions
on  ${\bf R}$ and the measures under consideration are infinite powers of the corresponding measures are essential. Indeed, let $\mu$ be a linear  Gaussian measure on  ${\bf R}$ whose density distribution function has the form $f(x)=\frac{1}{\sqrt{2\pi}}e^{-\frac{x^2}{2}}~(x \in {\bf R})$. Let $\delta_x$ be a Dirac measure defined on the Borel $\sigma$-algebra of subsets of  ${\bf R}$ and concentrated at $x~(x \in {\bf R})$. Let $D$ be a subset of ${\bf R^N}$ defined by
$$
D=\{ (x_k)_{k \in N}: \lim_{n \to \infty} \frac{\sum_{k=1}^nx_k}{n}=0\}.
$$
It is obvious that $D$ is a Borel subset of ${\bf R^N}$.

For $(x_k)_{k \in N} \in D$ we set $\mu_{(x_k)_{k \in N}}=\prod_{k \in N}\delta_{x_k}$\footnote{Note that $\prod_{k \in N}\delta_{x_k}=\delta_{(x_k)_{k \in N}}.$ }. Let us consider the family of Borel probability measures $\{ \mu^N\} \cup \{ \mu_{(x_k)_{k \in N}} : (x_k)_{k \in N} \in D\}.$  It is obvious that  it is an orthogonal  family of Borel product-measures for which Theorem 4.14 fails. Indeed, assume the contrary and  let $\{ C\} \cup \{ C_{(x_k)_{k \in N}} : (x_k)_{k \in N} \in D\}$ be  such a partition of ${\bf R^N}$ into elements of the $\sigma$-algebra $S_0(R^N)=\cap_{(x_k)_{k \in N} \in D}\mbox{dom}(\overline{\mu}_{(x_k)_{k \in N}}) \cap \mbox{dom}(\overline{\mu^N})$  that $\overline{\mu}_{(x_k)_{k \in N}}(C_{(x_k)_{k \in N}})=1$ for $(x_k)_{k \in N} \in D$ and $\overline{\mu^N}(C)=1$.
Since $(x_k)_{k \in N} \in C_{(x_k)_{k \in N}}$ for each $(x_k)_{k \in N} \in D$ we deduce that $D \cap C=\emptyset$. This implies that $\overline{\mu^N}(C)\le \overline{\mu^N}({\bf R^N} \setminus D)=0$ because by the strong law of large numbers we have that $\overline{\mu^N}(D)=1$. The latter relation is a contradiction and Remark 4.17 is proved.
\end{rem}


\begin{thebibliography}{9}



\bibitem{CamMar44}\textit{Cameron R.H., Martin W.T.} On Transformations of Wiener integrals under translations // Ann.of Math.-1944.- \textbf{45}.- P. 386-396.


\bibitem{Feld58}\textit{Feldman J.}  Equivalence and orthogonality of Gaussian processes // Pacific J. Math. - 1958.-{\textbf 8}.- P. 699--708.



\bibitem{GKPP14} \textit{ Gill T., Kirtadze A., Pantsulaia G., Plichko A.} Existence and uniqueness of translation invariant measures in separable Banach spaces // Funct. Approx. Comment. Math.-2014.- {\textbf 50(2)}.-P. 401--419.


\bibitem{GPZ12} \textit{Gill T.L., Pantsulaia G.R., and Zachary W.W.}  Constructive Analysis In Infinitely Many Variables // Communications in Mathematical Analysis.-2012.-{\textbf 13 (1)}.-P. 107-141.


\bibitem{Gir59} \textit{Girsanov I.V., Mityasin B.S.} Quasi-invariant
 measures and linear topological spaces (in Russian)//
Nauchn. Dokl. Vys. Skol.-1959.-{\textbf 2}.-P. 5-10.


\bibitem{Gren50}\textit{Grenander Ulf.}  Stochastic processes and statistical
  inference // Ark. Mat.-1950.- {\textbf 1}.-P. 195--277.







\bibitem{Hill71}\textit{Hill D. G. B.}  $\sigma$-finite invariant measures on infinite product spaces // Trans. Amer. Math. Soc.-1971.- {\textbf 153}.- P. 347--370.


\bibitem{Kak48}\textit{Kakutani S.} On equivalence of infinite product measures // Ann. Math.-
1948.- {\textbf 4 (9)}.- P. 214-224.

\bibitem{Khar84}\textit{ Kharazishvili A.B.} On invariant measures in the Hilbert space (in Russian) // Bull. Acad. Sci.Georgian SSR.-1984.-\textbf{ 114(1)}.-P. 41--48.


\bibitem{KuNi74}\textit{ Kuipers L.,  Niederreiter H.}  Uniform distribution of sequences.-New York etc.:
John Wiley \& Sons, 1974.



\bibitem{Lep72}\textit{LePage R.D., Mandrekar V.}  Equivalence-singularity dichotomies from zero-one laws // Proc. Amer. Math. Soc.-1972.- {\textbf 31}.-P. 251--254.


\bibitem{Moore65}\textit{ Moore C.C.}  Invariant measures on product spaces // Proc. Fifth Berkeley Sympos. Math. Statist. and Probability.-1965-1966.- Vol. II: Contributions to Probability Theory, Part 2.- P. 447--459



\bibitem{Pan02}\textit{Pantsulaia G.} Duality of measure and category in infinite-dimensional separable Hilbert space $l\sb 2$ //
Int. J. Math. Math. Sci.-2002.-\textbf{30(6)}.-P. 353-363.
\MR{1904675}


\bibitem{Shim75}\textit{Shimomura, Hiroaki.} An aspect of quasi-invariant measures on $R\sp{\infty }$ // Publ. Res. Inst. Math. Sci.-1975/76.-
\textbf{11(3)}.- P.  749--773.


\bibitem{Skor74}\textit{Skorokhod A.V.} Integration in Hilbert space (in Russian).- Moscow, 1974.
English transl.: Springer, 1975.





\bibitem{Sud59} \textit{Sudakov V.N.}Linear sets with quasi-invariant measure(in Russian)// Dokl. Akad. Nauk SSSR.-1959.- \textbf{127}.-P. 524-525.



\bibitem{Ver66}\textit{Ver\v sik A. M.,} Duality in the theory of measure
 in linear spaces (in Russian) //Dokl. Akad. Nauk SSSR.-1966.-\textbf{ 170}.-P. 497--500.


\bibitem{Xia72}\textit{ Xia D.X.} Measure and integration on infinite-dimensional spaces.-New York.: Academic
Press, 1972.


\bibitem{Yam80}\textit{Yamasaki Y.} Translationally invariant measure on the infinite-dimensional vector space // Publ. Res. Inst. Math. Sci.-1980. - \textbf{16(3)}.-P. 693--720.


\bibitem{ZerPanSaat13}\textit{ Zerakidze Z., Pantsulaia G.,  Saatashvili G.}  On the separation
problem for a family of Borel and Baire $G$-powers of
shift-measures on $\mathbb{R}$ // Ukrainian Math. J. -2013.- {\textbf 65 (4)}.- P. 470--485.

\end{thebibliography}
\end{document}